\documentclass[11pt]{amsart}
\usepackage[english]{babel}

\usepackage[all]{xy}
\usepackage{amssymb,latexsym}
\usepackage{amssymb}
\usepackage{latexsym}
\usepackage{graphicx}
\usepackage{amssymb}
\usepackage{amsfonts}
\usepackage{amsmath}
\usepackage{latexsym}
\usepackage{amsthm}
\usepackage{xypic}
\usepackage{mathrsfs}
\usepackage{stmaryrd}
\usepackage{enumerate}
\usepackage{psfrag}

\begin{document}

\title{Closed and transitive transformation groups of a surface}

\author{Ferry H. Kwakkel}

\email{ferrykwakkel@gmail.com}


\begin{abstract}
The purpose of this paper is to survey the structure of closed and transitive transformation groups acting on a closed surface.
In particular, we prove a number of relations between groups acting on the sphere that contain the rotation group, 
together with a diagram of how these groups are connected. In addition, we describe transformation groups
acting on the torus and higher genus surfaces. 
\end{abstract}

\maketitle

\newtheorem{thm}{Theorem}
\newtheorem{cor}[thm]{Corollary}
\newtheorem{counter}[thm]{Counterexample}

\newtheorem*{thma}{Theorem A}
\newtheorem*{thmb}{Theorem B}
\newtheorem*{thmc}{Theorem C}
\newtheorem*{thm1}{Theorem}
\newtheorem*{problem}{Problem}
\newtheorem*{problemb}{Problem B}

\newtheorem*{mostow}{Classification of Lie groups}

\newtheorem*{conj}{Conjecture}

\newtheorem*{conja}{Conjecture A}
\newtheorem*{conjb}{Conjecture B}

\newtheorem*{kernel}{Problem}
\newtheorem*{kernelconj}{Kernel subgroup conjecture}

\newtheorem*{antlambda}{Ant Lambda Conjecture}

\theoremstyle{plain}
\newtheorem{lem}{Lemma}[section]
\newtheorem{prop}{Proposition}[section]
\newtheorem*{propo}{Proposition}
\newtheorem*{coro}{Corollary}

\theoremstyle{definition}
\newtheorem*{defn}{Definition}
\newtheorem*{rem}{Remark}
\newtheorem*{notation}{Notation}
\newtheorem{exc}[thm]{Exercise}
\newtheorem{exe}[thm]{Exercise}
\newtheorem{open}[thm]{Open Problem}

\newcommand{\Mobius}{\tu{M\"obius}}
\newcommand{\Ant}{\tu{Ant}}
\newcommand{\RP}{\R\mathbb{P}^2}
\newcommand{\Lin}{\tu{Lin}}
\newcommand{\Fin}{\tu{Fin}}

\newcommand{\Rotation}{\tu{Rotation}}
\newcommand{\MCG}{\tu{MCG}}
\newcommand{\Trans}{\tu{Trans}}
\newcommand{\Torelli}{\tu{Torelli}}

\newcommand{\Skew}{\tu{Skew}}
\newcommand{\Prod}{\tu{Prod}}

\newcommand{\trace}{\tu{trace}}

\numberwithin{equation}{section}
\renewcommand{\Im}{\operatorname{Im}}

\newcommand{\di}{\partial}
\newcommand{\dibar}{\bar\partial}
\newcommand{\hookra}{\hookrightarrow}
\newcommand{\ra}{\rightarrow}
\newcommand{\hra}{\hookrightarrow}
\newcommand{\imply}{\Rightarrow}
\def\lra{\longrightarrow}
\newcommand{\wc}{\underset{w}{\to}}
\newcommand{\tu}{\textup}

\newcommand{\diam}{\operatorname{diam}}

\newcommand{\stab}{\operatorname{Stab}}
\newcommand{\centr}{\operatorname{Centr}}

\newcommand{\dist}{\operatorname{dist}}
\newcommand{\Hdist}{\operatorname{H-dist}}
\newcommand{\cl}{\operatorname{cl}}
\newcommand{\inter}{\operatorname{int}}
\renewcommand{\mod}{\operatorname{mod}}
\newcommand{\card}{\operatorname{card}}
\newcommand{\tl}{\tilde}
\newcommand{\ind}{ \operatorname{ind} }
\newcommand{\Dist}{\operatorname{Dist}}
\newcommand{\len}{\operatorname{\l}}

\newcommand{\ctg}{\operatorname{ctg}}
\newcommand{\arcctg}{\operatorname{arcctg}}
\newcommand{\orb}{\operatorname{orb}}
\newcommand{\HD}{\operatorname{HD}}
\newcommand{\supp}{\operatorname{supp}}
\newcommand{\id}{\operatorname{id}}
\newcommand{\length}{\operatorname{length}}
\newcommand{\area}{\operatorname{area}}
\newcommand{\dens}{\operatorname{dens}}
\newcommand{\meas}{\operatorname{meas}}
\newcommand{\distM}{\operatorname{dist}_{Mon}} 
\newcommand{\per}{\operatorname{per}}
\newcommand{\Id}{\operatorname{Id}}

\newcommand{\PG}{\mathfrak{P}}

\renewcommand{\d}{{\diamond}}

\newcommand{\Dil}{\operatorname{Dil}}
\newcommand{\Ker}{\operatorname{Ker}}
\newcommand{\tg}{\operatorname{tg}}
\newcommand{\codim}{\operatorname{codim}}
\newcommand{\isom}{\approx}
\newcommand{\comp}{\circ}
\newcommand{\esssup}{\operatorname{ess-sup}}
\newcommand{\diff}{\operatorname{Diff}}

\newcommand{\SLa}{\underset{\La}{\Subset}}

\newcommand{\const}{\mathrm{const}}
\def\loc{{\mathrm{loc}}}
\def\fib{{\mathrm{fib}}}

\newcommand{\eps}{{\varepsilon}}
\newcommand{\epsi}{{\epsilon}}
\newcommand{\veps}{{\varepsilon}}
\newcommand{\De}{{\Delta}}
\newcommand{\de}{{\delta}}
\newcommand{\la}{{\lambda}}
\newcommand{\La}{{\Lambda}}
\newcommand{\si}{{\sigma}}
\newcommand{\Si}{{\Sigma}}
\newcommand{\Om}{{\Omega}}
\newcommand{\om}{{\omega}}

\newcommand{\al}{{\alpha}}
\newcommand{\ba}{{\mbox{\boldmath$\alpha$} }}
\newcommand{\bbe}{{\mbox{\boldmath$\beta$} }}
\newcommand{\bk}{{\boldsymbol{\kappa}}}
\newcommand{\bg}{{\boldsymbol{\gamma}}}

\newcommand{\bare}{{\bar\eps}}

\newcommand{\Ray}{{\mathcal R}}
\newcommand{\Eq}{{\mathcal E}}
\newcommand{\PR}{PR}

\newcommand{\AAA}{{\mathcal A}}
\newcommand{\BB}{{\mathcal B}}
\newcommand{\CC}{{\mathcal C}}
\newcommand{\DD}{{\mathcal D}}
\newcommand{\EE}{{\mathcal E}}
\newcommand{\EEE}{{\mathcal O}}
\newcommand{\II}{{\mathcal I}}
\newcommand{\FF}{{\mathcal F}}
\newcommand{\GG}{{\mathcal G}}
\newcommand{\JJ}{{\mathcal J}}
\newcommand{\HH}{{\mathcal H}}
\newcommand{\KK}{{\mathcal K}}
\newcommand{\LL}{{\mathcal L}}
\newcommand{\MM}{{\mathcal M}}
\newcommand{\NN}{{\mathcal N}}
\newcommand{\OO}{{\mathcal O}}
\newcommand{\PP}{{\mathcal P}}

\newcommand{\PPP}{{\mathbb P}}

\newcommand{\QQ}{{\mathcal Q}}
\newcommand{\QM}{{\mathcal QM}}
\newcommand{\QP}{{\mathcal QP}}
\newcommand{\QL}{{\mathcal Q}}

\newcommand{\RR}{{\mathcal R}}
\renewcommand{\SS}{{\mathcal S}}
\newcommand{\TT}{{\mathcal T}}
\newcommand{\TTT}{{\mathcal P}}
\newcommand{\UU}{{\mathcal U}}
\newcommand{\VV}{{\mathcal V}}
\newcommand{\WW}{{\mathcal W}}
\newcommand{\XX}{{\mathcal X}}
\newcommand{\YY}{{\mathcal Y}}
\newcommand{\ZZ}{{\mathcal Z}}

\newcommand{\A}{{\mathbb A}}
\newcommand{\C}{{\mathbb C}}
\newcommand{\bC}{{\bar{\mathbb C}}}
\newcommand{\D}{{\mathbb D}}
\newcommand{\Hyp}{{\mathbb H}}
\newcommand{\J}{{\mathbb J}}
\newcommand{\Ll}{{\mathbb L}}
\renewcommand{\L}{{\mathbb L}}
\newcommand{\M}{{\mathbb M}}
\newcommand{\N}{{\mathbb N}}
\newcommand{\Q}{{\mathbb Q}}
\newcommand{\R}{{\mathbb R}}
\newcommand{\T}{{\mathbb T}}
\newcommand{\V}{{\mathbb V}}
\newcommand{\U}{{\mathbb U}}
\newcommand{\W}{{\mathbb W}}
\newcommand{\X}{{\mathbb X}}
\newcommand{\Z}{{\mathbb Z}}

\newcommand{\VVV}{{\mathbf U}}
\newcommand{\UUU}{{\mathbf U}}

\newcommand{\tT}{{\mathrm{T}}}
\newcommand{\tD}{{D}}
\newcommand{\hyp}{{\mathrm{hyp}}}

\newcommand{\f}{{\bf f}}
\newcommand{\g}{{\bf g}}
\newcommand{\h}{{\bf h}}
\renewcommand{\i}{{\bar i}}
\renewcommand{\j}{{\bar j}}

\renewcommand{\k}{\kappa}

\def\Bf{{\mathbf{f}}}
\def\Bg{{\mathbf{g}}}
\def\BG{{\mathbf{G}}}
\def\Bh{{\mathbf{h}}}
\def\Bv{{\mathbf{v}}}
\def\Bz{{\mathbf{z}}}
\def\Bx{{\mathbf{x}}}
\def\By{{\mathbf{y}}}

\def\BH{{\mathbf{H}}}
\def\BF{{\mathbf{F}}}
\def\BS{{\mathbf{S}}}

\def\BT{{\mathbf{T}}}
\def\Bj{{\mathbf{j}}}
\def\Bphi{{\mathbf{\Phi}}}
\def\BPsi{{\boldsymbol{\Psi}}}
\def\BPhi{{\boldsymbol{\BPhi}}}
\def\B0{{\mathbf{0}}}
\def\BU{{\mathbf{U}}}
\def\BV{{\mathbf{V}}}
\def\BR{{\mathbf{R}}}
\def\BG{{\mathbf{G}}}
\newcommand{\Comb}{{\it Comb}}
\newcommand{\Top}{{\it Top}}
\newcommand{\QC}{\mathcal QC}
\newcommand{\Def}{\mathcal Def}
\newcommand{\Teich}{\mathcal Teich}
\newcommand{\PPL}{{\mathcal P}{\mathcal L}}
\newcommand{\Jac}{\operatorname{Jac}}
\renewcommand{\DH}{\operatorname{DH}}
\newcommand{\Homeo}{\operatorname{Homeo}}
\newcommand{\AC}{\operatorname{AC}}
\newcommand{\Dom}{\operatorname{Dom}}

\newcommand{\Aff}{\operatorname{Aff}}
\newcommand{\Euc}{\operatorname{Euc}}
\newcommand{\MobC}{\operatorname{M\ddot{o}b}({\mathbb C}) }
\newcommand{\PSL}{ {\tu{PSL}} }

\newcommand{\PGL}{ {\tu{PGL}} }
\newcommand{\PSO}{ {\tu{PSO}} }

\newcommand{\Cl}{ {\tu{Cl}} }
\newcommand{\Int}{ {\tu{Int}} }

\newcommand{\SL}{ {\tu{SL}} }
\newcommand{\CP}{ {\mathbb{CP}} }

\newcommand{\Mob}{\textup{M\"ob} }
\newcommand{\Rot}{\textup{Rot} }
\newcommand{\homeo}{\textup{Homeo} }
\newcommand{\diffeo}{\textup{Diffeo} }

\newcommand{\GL}{ \operatorname{GL} }
\newcommand{\SO}{ \operatorname{SO} }

\newcommand{\OOO}{ \operatorname{O} }

\newcommand{\hf}{{\hat f}}
\newcommand{\hz}{{\hat z}}
\newcommand{\hM}{{\hat M}} 

\renewcommand{\lq}{``}
\renewcommand{\rq}{''}


\catcode`\@=12

\def\Empty{}
\newcommand\oplabel[1]{
  \def\OpArg{#1} \ifx \OpArg\Empty {} \else
   \label{#1}
  \fi}

%

\long\def\realfig#1#2#3#4{
\begin{figure}[htbp]
\centerline{\psfig{figure=#2,width=#4}}
\caption[#1]{#3}
\oplabel{#1}
\end{figure}}

\numberwithin{figure}{section}

%

\newcommand{\comm}[1]{}
\newcommand{\comment}[1]{}

\setcounter{tocdepth}{2}

\section{Introduction and statement of results}\label{section_intro}

\subsection{History and definitions}

Let $\homeo(M)$ be the group of orientation preserving homeomorphisms of a closed orientable topological surface $M$, with $M$ either the sphere $S^2$, the torus $\T^2$, or a higher 
genus surface $S$ of genus at least two. The uniform topology defined on $\homeo(M)$, or a subgroup of $\homeo(M)$, is the topology induced by the topological 
distance $d_{C^0}$ between two homeomorphisms $f,g \in \homeo(M)$, defined by
\begin{equation}
d_{C^0}(f,g) = \sup_{p \in M} d(f(p), g(p)) = \max_{p \in M} d(f(p), g(p)),
\end{equation}
with $d(\cdot, \cdot)$ the metric defined on $M$, which is the standard spherical metric inducing the Lebesgue measure in the case of the sphere $S^2$. The topology on $\homeo(M)$ 
is defined to be the uniform topology. A subgroup $G$ of $\homeo(M)$ is said to be {\em closed} if it is a closed subset of $\homeo(M)$ as a topological subspace in the uniform 
topology, and $G$ is said to be {\em transitive} if for any two given $p ,q \in M$, there exists a $h \in G$, such that $h(p) = q$. A closed transitive subgroup is said to be {\em minimal}, 
respectively {\em maximal}, if the group is minimal, respectively maximal, relative to inclusion of subgroups, with respect to the property of being a closed and transitive proper 
subgroup of $\homeo(M)$. Since each such surface admits the group of area-preserving homeomorphisms homotopic to the identity which is a closed and transitive maximal proper subgroup 
of the group of homeomorphisms homotopic to the identity of that surface, see Theorem A below, it follows that for every closed surface $M$, there exists at least one maximal closed 
and transitive subgroup other than the full group of homeomorphisms of $M$ homotopic to the identity. The general problem presented is to classify the closed 
and transitive subgroups, modulo conjugation, of the homeomorphism group $\homeo(M)$ of a closed surface and find its maximal subgroups. This problem relates to 
the classification of Lie group actions on a closed surface, initiated by classical works of Lie~\cite{lie} and further classified by Mostow~\cite{mostow}. Classifying closed and 
transitive subgroups of the homeomorphism group of a closed manifold is a particular problem in the study of general homeomorphism groups of manifolds, see~\cite{fisher} by Fisher.

In dimension one, subgroups of the homeomorphism group of the circle have been studied by Ghys~\cite{ghys}, and in particular closed and transitive subgroups of the circle 
by Giblin-Markovic~\cite{MG}, answering an interesting question posed in~\cite{ghys}. The results in~\cite{MG} include that a closed and transitive subgroup of $\homeo(\mathbb{S}^1)$, 
with $\mathbb{S}^1$ the circle, that properly extends the rotation group $\Rot(\mathbb{S}^1)$ is either the group $\Mobius(\mathbb{S}^1)$ of M\"obius transformations, or one of its cyclic 
covers, or the full homeomorphism group, and one of its cyclic covers. In particular, the group $\Mobius(\mathbb{S}^1)$ acting on the circle is a maximal subgroup in $\homeo(\mathbb{S}^1)$. 

\subsection{Statement of results}

In what follows, denote $\R\PPP^2$ the projective plane, $\Lin(\R\PPP^2)$ the group of projective mappings, $\Ant(S^2)$ the centralizer of $\homeo(S^2)$ with respect to the 
antipodal action on the sphere, $\Mobius(S^2)$ the group of M\"obius transformations, and $\Homeo_{\lambda}(S^2) \subset \homeo(S^2)$ and $\Ant_{\lambda}(S^2) \subset \Ant(S^2)$ 
the corresponding subgroups of area-preserving homeomorphisms. A closed subgroup $G \subseteq \homeo(S^2)$, with $S^2$ the two-sphere, that properly extends the rotation 
group is called a {\em homogeneous transformation group}. We present the following diagram of homogeneous groups.

\begin{displaymath} 
\xymatrix{ 
& ~ & \Mobius(S^2) \ar@/^/[ddr] & ~   \\
& ~ & \homeo_{\lambda}(S^2) \ar@/^/[dr]_{\star} & ~ \\
\Rot(S^2) \ar@/^/[urr] \ar@/_/[drr] \ar@/^/[uurr]^\star \ar@/_/[ddrr]_\star \ar@//[r] & \Ant_{\lambda}(S^2) 
\ar@//[rr]^{\star} \ar@/^/[ur] \ar@/_/[dr]^{\star} & ~ & \Homeo(S^2) \\
& ~ & \Ant(S^2) \ar@/_/[ur]^{\star} \\
& ~ & \Lin(S^2) \ar@//[u] \ar@/_/[uur] & ~  \\ }
\end{displaymath}

Each proper subgroup $G \subset \homeo(S^2)$ appearing in the diagram is a group satisfying (a combination of) symmetry conditions, with the symmetry conditions 
being (i) preserving circles ($\Mobius(S^2)$), (ii) preserving geodesics ($\Lin(S^2)$), (iii) preserving the antipodal action $(\Ant(S^2)$) and (iv) preserving 
area ($\homeo_{\lambda}(S^2)$). For example, the group $\Rot(S^2)$ preserves all these symmetries and the group $\Ant_{\lambda}(S^2)$ is the group that satisfies both symmetries (iii) 
and (iv), but not (i) and (ii). An arrow $G \lra H$ between two groups $G, H$ in the diagram is defined to be {\em complete} if $G$ is maximal in $H$, in the sense that any closed 
group $G \subset K \subseteq H$ that has the property that all the symmetries that $G$ has, but that $H$ does not have, are not respected by $K$, then it holds that $K = H$. 
We prove the following set of implications in terms of the subgroups in the diagram.

\begin{thma}
Given the homogeneous transformation groups $G \subseteq \homeo(S^2)$ in the diagram, each arrow marked with $\star$ is complete and conversely each arrow not present 
between two groups in the diagram, other than $\Rot(S^2) \lra \homeo(S^2)$, corresponds to an intersection of groups already present in the diagram.
\end{thma}

A similar result in the setting of area-preserving homeomorphisms has been obtained recently by Le Roux in~\cite{roux}. We state the following.

\begin{problemb}
Classify all homogeneous transformation groups and their relations.
\end{problemb}

Current work of the author includes a careful study of the properties of the isotopy subgroup of a homogeneous transformation group in $\homeo(S^2)$,
and further relations in the diagram. 

Further, we remark the following. Call a closed and transitive group $G \subset \homeo(M)$, with $M$ the sphere, torus, or higher genus surface, an {\em exotic group} if 
the group $G$ does not contain a continuous arc of homeomorphisms. It is an interesting problem to determine whether there exist such exotic closed and transitive groups. 
In the remainder, we first discuss several results about Lie group actions on a surface. After this section, we proceed with the main result about the spherical groups, and close 
the paper with a list of groups in the case of the torus and higher genus surfaces.

\subsection{Lie group actions on a surface}

Since Lie groups, both finite-dimensional and infinite-dimensional, are particularly well understood, a correspondence between homeomorphism groups 
and Lie group actions is of interest, and we detail how the groups described in this paper can be given the interpretation of a Lie group, 
modulo passing to a finite degree regular cover of the surface on which the closed and transitive group is defined. In the case of the sphere,
the group $\Ant(S^2) \subset \homeo(S^2)$ is the group of homeomorphisms acting on the projective plane $\R\PPP^2$ and so with the 
subgroups of $\Ant(S^2)$. The pointwise classification of Lie groups by Mostow has been used by Belliart in~\cite{belliart} to classify the 
finite-dimensional Lie group actions of a surface without fixed points, which are the following
\begin{enumerate}
\item spherical case: the orthogonal action $\SO(3, \R)$, the complex-projective action $\PGL(2, \C)$, and the real-projective action $\PGL(3, \R)$,
\item torus case: the circle-action $(x,y) \mapsto (x + \alpha, y) \mod \Z^2$,
\item higher genus case: none.
\end{enumerate}
The possible pointwise $\GL(2, \R)$ actions of a mapping is described by the Lie subgroups of the Lie group $\GL(2, \R)$, namely
\begin{enumerate}
\item[(i)] the orthogonal group $\tu{O}(2,\R) \subset \GL(2, \R)$, 
\item[(ii)] the special linear group $\SL(2,\R) \subset \GL(2, \R)$, 
\item[(iii)] the Borel subgroup $\tu{U}(2, \R) \subset \GL(2,\R)$ of upper triangular matrices.
\end{enumerate}
On a given surface $M$, after passing to the $C^0$-closure, these pointwise Lie group actions integrate globally to (i) the finite-dimensional group of conformal homeomorphisms, 
(ii) the infinite-dimensional group of area-preserving homeomorphisms, and (iii) the infinite-dimensional group of homeomorphisms that fiber over the circle $\mathbb{S}^1$, 
in other words, the skew-product homeomorphisms of the form $g(x,y) = (\varphi(x), \psi(x,y))$. 

\subsection*{Acknowledgement}

The author worked on the main ideas presented in this paper during a two year research position at IMPA, Rio de Janeiro, Brazil, in the years 2010 and 2011, and wishes to 
thank these institutions for their hospitality. Further, he would like to thank Etienne Ghys for a useful discussion about the ideas in this paper at IMPA, and Fabio Tal 
for several useful discussions.

\section{The case of the sphere}

Closed and transitive subgroups of the homeomorphism group of the sphere that do not arise as extensions of the 
rotation group are called inhomogenous groups. These groups are equivalently characterized as closed and transitive groups that do not contain a compact 
and transitive kernel group. It is a natural question whether or not there exist such inhomogeneous groups on the sphere. On the torus, the group of Hamiltonian 
homeomorphisms is an example of such a inhomogeneous group. The homogeneous groups by definition, and the known 
inhomogeneous groups on the sphere, torus, and higher genus surfaces, are known to contain isotopies. Our focus in this section will be to consider homogeneous groups.

\subsection{Arrows in the finite-dimensional case}\label{sec_prelim}

The purpose of this section is to prove completeness of the arrows marked with $\star$ in the diagram involving the finite-dimensional transformation groups.
Geodesics on the sphere are great circles, where a circle on the sphere is the intersection of a Euclidean plane with the round sphere embedded in $\R^3$, 
and this circle is a great circle if the plane passes through the origin of the sphere.  We use the notation whereby the north-pole and south-pole on the sphere is 
indicated by $0$ and $\infty$ respectively. Further, given a point $p \in S^2$, denote $-p \in S^2$ its antipodal point on the sphere.

\subsubsection{The rotation group}\label{section_rotation_group}

The sphere $S^2$ can be defined as the set of points in $\R^3$ with unit distance from the origin. The group $\SO(3, \R)$ acts by rotations, and hence by isometries, on the sphere $S^2$ 
and we have that $\SO(3, \R) \cong \Rot(S^2)$. It is a classical result by Ker\'ekj\'art\'o~\cite{kerek}, that a compact group $G \subset \homeo(S^2)$ is topologically conjugate to a subgroup 
of $\Rot(S^2)$. Furthermore, since compact subgroups of $\SO(3,\R)$ are classified into either (i) finite groups, (ii) $\SO(2,\R)$, or (iii) $\SO(3,\R)$, we have that, up to conjugation, 
$\Rot(S^2)$ is the unique compact transitive subgroup of $\homeo(S^2)$ and thus $\Rot(S^2)$ is minimal as a compact and transitive subgroup of $\homeo(S^2)$.

\subsubsection{The real-linear action}

The real-projective plane $\R\PPP^2$ is defined as the sphere modulo the antipodal action $p \sim q$ if and only if $p = -q$. The group $\PGL(3,\R)$ acts as the 
centralizer of $\GL(3, \R)$ by the homothetic action. This action passes to the quotient $\R \PPP^2$. The action of an element $T \in \GL(3, \R)$
We have $\Lin(\R\PPP^2) \cong \PGL(3, \R)$ and $\Lin(\R\PPP^2) \subset \Homeo(\R\PPP^2)$ the group of orientation-preserving collineations, or projective transformations 
of $\R\PPP^2$. 

\begin{lem}\label{lem_area_linear}
We have that $\Lin(S^2) \cap \homeo_{\lambda}(S^2) = \Rot(S^2)$.
\end{lem}

\begin{proof}
Let $h$ be linear and area-preserving. Take two perpendicular geodesics, dividing the sphere into four equal quarters. These two perpendicular geodesics have to be sent to two 
perpendicular geodesics, since otherwise the area of these quarters is not preserved by $h$. Since this holds for each pair of perpendicular geodesics, this implies that $h$ is conformal. 
Since a conformal and linear mapping is an isometry by Lemma~\ref{lem_conf_lin} below, the proof is complete.
\end{proof}

\begin{prop}\label{prop_rot_lin}
The arrow $\Rot(S^2) \lra \Lin(S^2)$ is complete.
\end{prop}

By the singular value decomposition of matrices, every element $T \in \GL(3,\R)$ can be written as $R_1 D R_2$, with $R_1, R_2 \in \SO(3,\R)$ 
and $D \in \GL(3,\R)$ a diagonal matrix with eigenvalues $\lambda_1, \lambda_2, \lambda_3 \in \R^+$ on the diagonal. 

\begin{lem}\label{lem_pgl_three_decomp}
Given $T \in \PGL(3, \R)$ and $R_1, R_2 \in \SO(3,\R)$, where $R_1,R_2,T$ correspond to $r_1,r_2,h \in \Lin(S^2)$, with $r_1, r_2 \in \Rot(S^2)$. Then $R_1 T R_2$ corresponds 
to $r_1 h r_2 \in \Lin(S^2)$.
\end{lem}

\begin{proof}
The inclusion $\Rot(S^2) \subset \Lin(S^2)$ is induced by the inclusion $\SO(3,\R) \subset \GL(3,\R)$, which in turn induces the inclusion $\PSO(3, \R) \subset \PGL(3, \R)$. Since the group 
homomorphism $\GL(3,\R) \ra \Lin(S^2)$, becomes an isomorphism after projectivizing $\GL(3, \R) \ra \PGL(3, \R)$, the map $T \mapsto h_T$ in the quotient passes to a group isomorphism 
of $\PGL(3, \R) \ra \Lin(S^2)$, and elements of $\SO(3,\R) = \PSO(3, \R) \subset \PGL(3, \R)$ correspond to rotations on the sphere, the claim follows.
\end{proof}

\begin{lem}\label{lem_rotation_SL_three_full}
The group $\langle \SO(3,\R), A \rangle$ generated by $\SO(3,\R)$ and a non-orthogonal $A \in \SL(3,\R)$ equals $\SL(3,\R)$.
\end{lem}

\begin{proof}
Let $T \in \langle \SO(3,\R), A \rangle$ such that $T \notin \SO(3, \R)$, which we may assume is in diagonal form, and denote $\lambda_i>0$, with $i=1,2,3$ its eigenvalues. 
First, assume that $\lambda_2 =1$ for $T$, so that $\lambda_1 < 1 < \lambda_3$, where $\lambda := \lambda_3$ and $\lambda^{-1} := \lambda_1$. Iterating $T$, one gets 
$T_n := T^n$ for which $\lambda_1^n \ra 0$ and $\lambda_3^n \ra \infty$, and $\lambda_2^n = 1$ for all $n \in \N$. For fixed $n \in \N$, rotating $T_n$ along the axis defined 
by $v_2$ and taking compositions, one obtains $T$ with largest eigenvalue $t \in [1, \lambda^n]$ and smallest eigenvalue $t^{-1}$. Letting $n \ra \infty$, we thus obtain all 
possible combinations for $\lambda_1$ and $\lambda_3 = \lambda_1^{-1}$. To produce a non-orthogonal element $T$ for which $\lambda_2 =1$, take $T_0$ with general 
eigenvalues, where $\lambda_1 < 1 < \lambda_3$ and $\lambda_1 \leq \lambda_2 \leq \lambda_3$, rotate by a small amount along the axis defined by $v_2$, and take the 
inverse of $T_0$ with the rotated version, so that the eigenvalue corresponding to the direction $v_2$ is one, but the new eigenvalues lying in the plane spanned by $v_1$ 
and $v_3$ are not both one. This produces the desired non-orthogonal element $T$ with the properties as mentioned.
To produce any $T \in \SL(3,\R)$, take an element $T_1$ with eigenvalues $(\lambda_1, \lambda_2, \lambda_3) = (\lambda_1 , 1, \lambda_1^{-1})$ and compose with $T_2$ with 
eigenvalues $(1, \lambda_2, \lambda_2^{-1})$, so that $T_1 T_2$ has eigenvalues $(\lambda_1, \lambda_2, (\lambda_1 \lambda_2)^{-1})$. Pre- and postcomposing $T_1T_2$ with 
rotations, we obtain any $T \in \SL(3,\R)$ as desired.
\end{proof}

\begin{proof}[Proof of Proposition~\ref{prop_rot_lin}]
Since the groups $\SL(3, \R)$ and $\PGL(3, \R)$ are group isomorphic, combining Lemma~\ref{lem_pgl_three_decomp} with Lemma~\ref{lem_rotation_SL_three_full}, 
the claim follows.
\end{proof}

\subsubsection{The complex-linear action}

The complex-linear action on the sphere, here identified as the complex plane $\C$ compactified with a point $\infty$, where, taking the centralizer of $\SL(2, \C)$ by the homothetic action, 
one obtains the action $\PGL(2, \C)$, the M\"obius action, acting by homeomorphisms on the sphere. The group $\Mobius(S^2)$ is sharply $3$-transitive on $S^2$ and consists precisely 
of the conformal homeomorphisms of $S^2$.

\begin{lem}[Antipodal M\"obius transformations]\label{lem_conf_lin}
Antipodal M\"obius transformations are rotations, that is, $\Mobius(S^2) \cap \Ant(S^2) = \Rot(S^2)$.
\end{lem}

\begin{proof}
Take $h \in \Ant(S^2)$. Pre- and post-composing with a rotations, we may assume that $h(0) = 0$, and thus also $h(\infty) = (\infty)$. Consider the image $\gamma = h(\gamma_1)$ 
where $\gamma_1$ is the horizontal great circle passing through $1 \in S^2$. As $h \in \Mobius(S^2) \cap \Ant(S^2)$, $\gamma$ is again a great circle as it can not be contained in a single 
hemisphere. In particular, $\gamma \cap \gamma_1$ has at least two intersection points. Taking a rotation $r \in \stab_{\Rot}(0,\infty)$ rotating one of the points $\gamma \cap \gamma_1$ 
back to $1 \in S^2$, the homeomorphism thus obtained fixes $0,1,\infty$. Since the only M\"obius transformation fixing $0,1,\infty$ is the identity by sharp $3$-transitivity of the M\"obius 
group, the claim follows.
\end{proof}

\begin{lem}\label{lem_stab_two_isom}
For each $g \in G$, there exist $r_1, r_2 \in \Rot(S^2)$ and $h \in G_2$, such that $g = r_1 h r_2$.
\end{lem}

\begin{proof}
To prove the result, we claim that for every $g \in \homeo(S^2)$, there exist a pair of antipodal points $\{p,-p \}$ in $S^2$ with the property that $g(-p) = - g(p)$. The pairs $\{p,-p\}$ 
and $\{g(p), -g(p)\}$ can be brought by rotations $r_1, r_2 \in \Rot(S^2)$ back to $\{0,\infty \}$ respectively, to prove the desired result. To prove the claim, suppose that $g(-p) \neq - g(p)$ 
for all $p \in S^2$. Define a line field on the sphere $S^2$ as follows. For each $q \in S^2$, denote $p := g^{-1}(q)$ and let $\gamma_{q}$ be the geodesic
passing through $q = g(p)$ and $w=g(-p)$. The geodesic passing through $q$ is unique and assigns a line element at $q$ that depends continuously on the basepoint $q$ since antipodal 
points $\{p,-p\}$ are mapped by $g$ to points whose distance is bounded away from $0$ and $\pi$, for all $p \in S^2$ by continuity of $g$ combined with the assumption that $g$ 
does not send antipodal points to antipodal points. This line field is globally defined and everywhere continuous, which is impossible by the hairy ball theorem.
\end{proof}

\begin{lem}\label{lem_mob_area_intersect}
Area-preserving M\"obius transformations are rotations, that is, $\homeo_{\lambda}(S^2) \cap \Mobius(S^2) = \Rot(S^2)$.
\end{lem}

\begin{proof} 
By Lemma~\ref{lem_stab_two_isom}, a homeomorphism $h \in \Mobius(S^2)$ is a composition $h = r_1 g r_2$, with $r_1, r_2 \in \Rot(S^2)$ and $g \in G_2$. In this case $g$ is a 
northpole-southpole action along latitudes post-composed with a rotation. Since $h$ is area-preserving if and only if $g$ is area-preserving, which is the case if and only if $g$ is a 
pure rotation and has no proper northpole-southpole action, $h$ itself is a composition of rotations, and thus a rotation as required.
\end{proof}

\begin{prop}\label{prop_rot_mob_max}
The arrow $\Rot(S^2) \lra \Mobius(S^2)$ is complete.
\end{prop}

The proof of Proposition~\ref{prop_rot_mob_max} is a direct construction in line with the proof for the case of the linear group. Let $G = \langle \Rot(S^2), g \rangle$ denote the 
group generated by the rotation group $\Rot(S^2)$ and any $g \in \Mobius(S^2)$ which is not an isometry. In the following lemma, denote $\HH \subset S^2$ the upper hemisphere 
defined as the connected component containing $\infty$ of $S^2 \setminus \gamma$, with $\gamma \subset S^2$ the horizontal geodesic relative to $0$ and $\infty$.

\begin{lem}\label{lem_meridian_trans}
The group $G \subseteq \Mobius(S^2)$ contains a subgroup $H \subset G$ of transformations leaving invariant the northern hemisphere $\HH \subset S^2$, 
where $H$ contains the subgroup $\langle \Rot(\HH), h \rangle$, where $h \in \Mobius(\HH) \setminus \Rot(\HH)$.
\end{lem}

\begin{proof}
Since $\Ant(S^2) \cap \Mobius(S^2) = \Rot(S^2)$, adding $g \in \Mobius(S^2)$ which is not an isometry, $g$ has the property that there exist points $p \in S^2$ such that $h(-p) \neq -h(p)$ 
by Lemma~\ref{lem_conf_lin}. Take a rotation $r \in \Rot(S^2)$ around $0$ and $\infty$ and conjugate $r$ with $g$ to obtain an elliptic transformation $g_0 \in \Mobius(S^2)$ which is not 
an isometry. As $r$ acts along the leaves of the horizontal latitudes, and with $g \in \Mobius(S^2)$ sending circles to circles, there exists a horizontal latitude that $g$ sends to a great circle 
in $S^2$, and since the original horizontal latitudes foliates the sphere $S^2$, and the image foliation under $g$ consists of circles, there is a unique great circle in the image under $g$ of 
the latitudinal foliation. Since $r$ leaves each horizontal latitude invariant, the induced elliptic transformation $g_0 \in \Mobius(S^2)$ leaves invariant a unique great circle. 
Conjugating $g_0$ with a suitable rotation in $\Rot(S^2)$, we may assume that the invariant geodesic equals $\gamma \subset S^2$. In that case, $g_0$ leaves invariant the upper 
hemisphere $\HH$ containing $\infty$ and the induced action on $\gamma$ by $g_0$ is not isometric. The subgroup $\Rot(S^2)$ of rotations leaving invariant $\gamma$, which is 
isomorphic to $\SO(2, \R)$, acts isometrically on $\gamma$. Therefore, $G$ contains a subgroup $H \subset G$ acting on $\gamma$ with the required properties. 
\end{proof}

\begin{lem}\label{lem_mob_disk_max}
In the notation of the previous lemma, we have that $H := \langle \Rot(\HH), h \rangle = \Mobius(\HH)$, for any $h \in \Mobius(\HH)$ which is not an isometry.
\end{lem}

\begin{proof}
Indeed, project the hemisphere $\HH$ conformally on the unit disk $\D$ and denote again $\gamma = \partial \D$. The full M\"obius group $\Mobius(\D)$ is isomorphic to $\PGL(2,\R)$. 
In turn, the group $\PGL(2, \R)$ is isomorphic to $\SL(2, \R)$. The action $\PSO(2, \R)$ identifies in this isomorphism with $\SO(2, \R)$. Therefore, the statement reduces to the claim 
that $\SL(2, \R) = \langle \SO(2, \R), A \rangle$ for a non-orthogonal $A \in \SL(2, \R)$, which follows as in the case of $\SL(3,\R)$ readily from the singular value decomposition theorem.
\end{proof}

\begin{proof}[Proof of Proposition~\ref{prop_rot_mob_max}]
To produce the full group $\Mobius(S^2)$, we use Lemma~\ref{lem_meridian_trans} and the group $\Rot(S^2)$. First, an elliptic element is determined by an axis of rotation $\alpha$ and a 
rotation angle $\theta \in \mathbb{S}^1$. The axis $\alpha$ is the unique hyperbolic geodesic in the three-dimensional round ball enclosed by the sphere $S^2$ passing through two given 
points $p,q \in S^2$. Unless $p$ and $q$ are antipodal, in which case the corresponding elliptic element is contained in $\Rot(S^2)$, define $\gamma$ to be the unique great circle passing 
through $p$ and $q$. By Lemma~\ref{lem_meridian_trans}, there exists $g \in \Mobius(S^2)$ leaving invariant $\gamma$ and sending $p$ and $-p$ to $p$ and $q$ respectively. Conjugating 
the rotation $r_{\theta} \in \Rot(S^2) \subset \Mobius(S^2)$ by $g \in \Mobius(S^2)$ produces the desired elliptic element. A hyperbolic element in $\Mobius(S^2)$ is determined uniquely by 
the translation axis passing through two different points $p,q \in S^2$, the translation length and the rotation angle around the axis. Given a translation length $T$, 
by Lemma~\ref{lem_mob_disk_max}, there exists a hyperbolic transformation $\Mobius(\D)$ with translation length $T$ and whose translation axis is a geodesic passing through the origin 
of $\D$. Taking the double of this transformation, we obtain a hyperbolic element in $\Mobius(S^2)$ with an axis whose endpoints are antipodal, which we may assume upon a rotation, 
to be $\infty$ and $0$. Applying again Lemma~\ref{lem_meridian_trans}, we can find $g \in \Mobius(S^2)$ sending $0, \infty$ to $p,q$ respectively. Conjugating the original hyperbolic 
element with antipodal fixed points with $g$, one obtains a hyperbolic transformation in $\Mobius(S^2)$ with the required axis and translation length. Precomposing this transformation 
with a rotation of angle $\theta \in \mathbb{S}^1$ around $\infty$ and $0$, one obtains the required general hyperbolic (or loxodromic) element. Parabolic elements are constructed similarly 
by passing to the double and conjugating.
\end{proof}

\subsection{Arrows in the infinite-dimensional case}\label{sec_struc_diagram}

The purpose of this section is to prove completeness of the arrows marked with $\star$ in the diagram involving the infinite-dimensional transformation groups.
This will show how, beyond the arrows in the finite-dimensional case, perturbing the symmetries of one group leads in a constructive way to generating a much larger group. 

\begin{prop}\label{thm_area_max}
The area-preserving group $\Homeo_{\lambda}(S^2)$ is maximal in $\homeo(S^2)$.
\end{prop}

A good arc or good curve $\eta, \gamma \subset S^2$ is an arc or curve that has zero Lebesgue measure and a good disk is a closed topological disk whose boundary curve is 
good curve. Since measure is additive over disjoint sets, a generic, in the Hausdorff topology, curve is a good curve, and similarly with arcs. Consequently, arbitrarily 
close to a disk $D \subset S^2$, we can find a good disk with the same Lebesgue measure. To prove the result, use the Oxtoby-Ulam integration theorem~\cite{oxtoby}. 
An Oxtoby-Ulam measure is a finite Borel measure, that assigns zero measure to points and positive
measure to sets with interior. By the Oxtoby-Ulam integration theorem, given two Oxtoby-Ulam measures $\mu_1, \mu_2$ on $S^2$ with unit total mass, there exists a 
homeomorphism $h \in \homeo(S^2)$ such that $h_*(\mu_1) = \mu_2$. In our setup, we use $\mu_1 = \mu_2 = \lambda$, where $\lambda$ denotes the standard Lebesgue 
measure on the sphere. An alternative result by Oxtoby-Ulam, namely the Oxtoby-Ulam extension theorem, states that given a closed topological disk $D$ and a 
homeomorphism of the boundary $\partial D$, there exists an Lebesgue measure preserving homeomorphism of the disk $D$ onto itself with the prescribed boundary values.

An $\epsilon$-grid $\Gamma(\epsilon) \subset S^2$ on the sphere is defined the union of parallel meridians and parallel latitudes such that the complementary disks
of the grid in the sphere are cells of diameter at most $\epsilon$. Let $G$ be a closed subgroup of $\homeo(S^2)$ properly containing $\Homeo_{\lambda}(S^2)$.
The argument consists in showing that, given $h \in \homeo(S^2)$, and given $\epsilon>0$, the warped image grid $h(\Gamma(\epsilon))$ can be moved back to its 
original position using homeomorphisms in $G$. By a diagonal argument, we may assume that the homeomorphism we need to approximate is a homeomorphism
that preserves Lebesgue null sets, for example by taking diffeomorphisms or piecewise linear homeomorphisms, which are dense in $\homeo(S^2)$.
In what follows, a homeomorphism $h \in \homeo(S^2)$ is said to have trivial area gradient on a domain $U \subseteq S^2$ if there exists a constant $c \in (0, \infty)$ 
such that $\meas(h(X)) = c \meas(X)$, for each measurable $X \subset U$, and $h$ is said to have non-trivial area gradient on $U$ otherwise. 

\begin{lem}[Elementary move]\label{lem_elem_move}
For sufficiently small $\delta >0$, given a good disk $D \subset S^2$ of area at most $\delta$, and given two good arcs $\eta_1, \eta_2 \subset D$ having the same two endpoints 
in the boundary of $D$, there exist $g \in G$ such that $g$ is the identity on $S^2 \setminus D$ and $g(\eta_1) = \eta_2$.
\end{lem}

\begin{figure}[h]
\begin{center}
\psfrag{p_1}{$\infty$}
\psfrag{p_0}{$0$}
\psfrag{D}{$D$}
\psfrag{g_1}{$\eta_1$}
\psfrag{g_2}{$\eta_2$}
\psfrag{h_0}{$g$}
\includegraphics[scale=0.3]{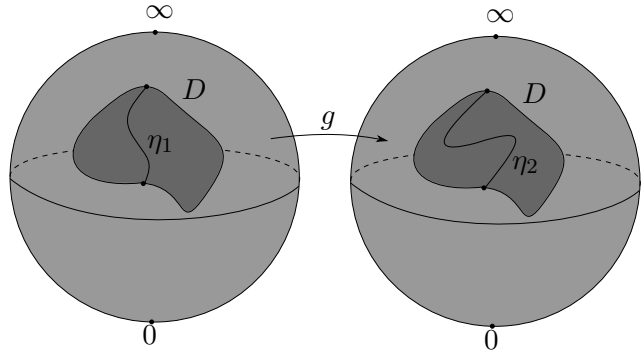}
\caption{The elementary move of Lemma~\ref{lem_elem_move}.}
\end{center}
\end{figure}

Lemma~\ref{lem_elem_move} is proved in several steps. 

\begin{lem}\label{lem_jordan_oxtoby}
Let $D_1, D_2 \subset S^2$ be two good disks of equal area. Then there exists $h \in \homeo_{\lambda}(S^2)$ such that $h(D_1) = D_2$.
\end{lem}

\begin{proof}
Let $\gamma_1$ and $\gamma_2$ be the two simple closed curves of measure zero comprising the boundary of the closed topological disks $D_1$ and $D_2$. 
By the Jordan curve theorem, the simple closed curves divides both $S^2 \setminus \gamma_1$ and $S^2 \setminus \gamma_2$ into precisely two connected components, 
and each point of $\gamma_1$ lies on the common boundary of each component, and similarly with $\gamma_2$.
Let $m_D$ denote the Lebesgue measure of both $D_1$ and $D_2$, which are equal in measure, and let $m_S$ denote the Lebesgue measure of $S^2 \setminus D_1$
and $S^2 \setminus D_2$, also equal in measure. Construct the Oxtoby-Ulam measure $\mu_1$ defined as twice the Lebesgue measure $2 \lambda$ on $S^2 \setminus D_1$
and the Lebesgue measure on $D_1$, normalized so that the entire sphere has unit area. Similarly, define the measure $\mu_2$ in the exact same manner but now on 
$S^2 \setminus D_2$ and $D_2$. By the Oxtoby-Ulam integration theorem, there exists a homeomorphism $h \in \homeo(S^2)$ that integrates 
the measure $\mu_1$ to $\mu_2$. In particular, due to the separation property of the simple closed curves that bound the disk, and the distinct weight given to the inside of 
the disk and the outside, combined with $h$ integrating $\mu_1$ to $\mu_2$, it follows that $h(D_1) = D_2$. Furthermore, combining that $h$ integrates $\mu_1$ to $\mu_2$
with the preservation of the weights by construction of the mapping, $h$ preserves the Lebesgue measure on a full measure set of $S^2$ and therefore 
$h \in \homeo_{\lambda}(S^2)$, as required.
\end{proof}

\begin{lem}\label{lem_jordan_oxtoby_2}
Given any good disk $D \subset S^2$, and given a pair of good arcs $\eta_1, \eta_2 \subset D$ with the same endpoints and dividing $D$ into two 
subdisks with equal corresponding area, $g \in \homeo_{\lambda}(S^2)$ supported on $D$ such that $g(\eta_1) = \eta_2$.
\end{lem}

\begin{proof}
Combining Lemma~\ref{lem_jordan_oxtoby} with the Oxtoby-Ulam extension theorem, given two good closed topological disks $D_1$ and $D_2$ with the same area,
and a homeomorphism $\phi$ between their boundary curves, there exists an area-preserving homeomorphism of $D_1$ onto $D_2$
with boundary values $\phi$. Given the disk $D$ and $\eta_1 \subset D$, define $\beta_1, \beta_2 \subset \partial D$ the two components of $\partial D$ minus 
the two endpoints of $\eta_1$. Denote $D_1 \subset D$ the closed disk with circumference $\beta_1 \cup \eta_1$ and denote $\bar{D}_1 \subset D$ the 
closed disk with circumference $\beta_1 \cup \eta_2$. Define the closed topological disks $D_2$ with boundary $\eta_1 \cup \beta_2$ 
and $\bar{D}_2$ with boundary $\eta_2 \cup \beta_2$ analogously. By the extension theorem, we can find an area-preserving homeomorphism
$h_1$ of $D_1$ onto $\bar{D}_1$ such that $h_1$ extends as the identity on the boundary arc $\beta_1$ and extends as an arbitrary homeomorphism
on the boundary arc $\eta_1$. Similarly, we can find an area-preserving homeomorphism $h_2$ sending $D_2$ on $\bar{D}_2$ with the property
that $h_2$ extends as the identity on $\beta_2$ and as the same homeomorphism on $\eta_1$ used to define $h_1$
Since the homeomorphisms $h_1$ and $h_2$ glue together as an area-preserving homeomorphism $h \colon D \ra D$ which is the identity
on the boundary $\partial D$, defining $h$ to be the identity on the exterior of $D$ in the sphere $S^2$, $h$ extends to an area-preserving
homeomorphism in $\homeo_{\lambda}(S^2)$ with the properties as claimed.
\end{proof}

\begin{lem}\label{lem_jordan_oxtoby_3}
Given a good disk $D$ and good disks $D_1, D_2 \subset D$ and $\bar{D}_1, \bar{D}_2$, there exists an area-preserving homeomorphism
$h \in \homeo_{\lambda}(S^2)$ supported on $D$ such that $h(D_1) = \bar{D}_1$ and $h(D_2) = \bar{D}_2$.
\end{lem}

\begin{proof}
First, applying the same argument as in Lemma~\ref{lem_jordan_oxtoby}, we can find a homeomorphism $g \in \homeo_{\lambda}(S^2)$
with the property that $g(D) = D$ and $g(D_i) = \bar{D}_i$ with $i=1,2$. Furthermore, we can choose a good thin annulus $A \subset D$,
with boundary curves $\beta_1$ and $\beta_2$ such that the outer boundary curve $\beta_2$ of $A$ coincides with the boundary curve of $D$, 
and such that $g$ fixes the annulus. Now, applying the extension theorem, we can find an area-preseving homeomorphism $h_0 \colon A \ra A$, 
such that $h_0$ extends the boundary values of $h$ on the inner boundary curve $\gamma_1$ induced by $h$ and extends as the identity of the 
outer boundary curve $\gamma_2$. Now define $h \in \homeo_{\lambda}(S^2)$ that equals $g$ on $D \setminus A$, and equals $h_0$ on $A$ and 
equals the identity on the rest of the sphere, we obtain the desired homeomorphism, as required.
\end{proof}

\begin{lem}\label{lem_jordan_oxtoby_4}
Given $\delta>0$, there exists a good disk $D \subset S^2$ of area less than $\delta$, and there exists $h \in G$ with support contained 
in $D$ with non-trivial area gradient.
\end{lem}

\begin{proof}
Indeed, take a good disk $D_0$, there exists $g \in G$, such that the area gradient of $g$ on $D_0$ has non-trivial area gradient. 
If not, using the rotation group $\Rot(S^2)$ and covering the sphere by copies of $D_0$, this would imply that each homeomorphism in $G$
preserves area, as on each copy the area is multiplied by a constant factor, with the area multiplier being constant on overlaps, this multiplier 
has to be constant on the entire sphere and therefore to equal $1$. In particular, we can find two disks $D_1, D_2 \subset D_0$ of equal area 
such that the image disks have unequal area. Applying Lemma~\ref{lem_jordan_oxtoby_3} to construct an $h_0 \in \homeo_{\lambda}(S^2)$ 
supported on $D_0$ that interchanges the disks $D_1$ and $D_2$, taking the homeomorphism $h \in G$, defined by $gh_0g^{-1}$ 
supported on $D:=h(D_0)$, which we may assume is a good disk as well, yields the desired homeomorphism. Since $D$ can be made as small
as desired, the claim follows.
\end{proof}

\begin{proof}[Proof of Lemma~\ref{lem_elem_move}]
Let $D \subset S^2$ be a good disk of area $\delta$ and let $\eta_1, \eta_2$ be two good arcs. By Lemma~\ref{lem_jordan_oxtoby_2}, 
it suffices to move area from one subdisk of $D \setminus \eta_1$ to another, by homeomorphisms supported on $D$, as the 
arc $\eta_2$ can be uniformized into the desired shape given the area balance. In order to move area from one subdisk to the other, first, 
by Lemma~\ref{lem_jordan_oxtoby_4}, given $D$, there exists $h \in G$ 
supported on $D$ with non-trivial area gradient. In particular, we can find two good disks $B_1, B_2 \subset D$ of equal area, such 
that $h(B_1)$ has area larger than $B_1$ and $h(B_2)$ has area smaller than $B_2$. Taking the disks $B_1$ and $B_2$ sufficiently small and 
applying Lemma~\ref{lem_jordan_oxtoby_3}, we can place one disk $B_1$ in one component $D_1$ of $D \setminus \eta_1$ and $B_2$ in the other 
component of $D \setminus \eta_1$, where $h(B_1)$ is contained in the same component as $D \setminus \eta_1$ and $h(B_2)$ contained in the 
other component of $D \setminus \eta_1$.

To produce a definite area increase, say from $D_1$ to $D_2$, first apply $h$ and consider the area balance. Either the area of $h(D_2)$ is larger than
the area of $D_2$, or else after uniformizing the arc $h(\eta_1)$ by a homeomorphism $g$ supported on $D$ using Lemma~\ref{lem_jordan_oxtoby_2}, 
which we may assume is a good arc, the arc is contained in $D_1$. In this case, the inverse of $gh$ creates a mass transport of $D_1$ into $D_2$.
To make this area increase definite, if necessary, first again Lemma~\ref{lem_jordan_oxtoby_2} supported on $D$ such that $\eta_1$ is sent to 
an arc $\eta_0$, which is a modified version of $\eta_1$. Namely, cut two small segments out of $\eta_1$, cut out one small segment out of the 
boundary of both $B_1$ and $B_2$, and attach in both cases a thin tube between the segment taken from $\eta_1$ and the boundary of the disk, 
to obtain a modified curve $\eta_0$. Choosing the tubes succintly, the area of either component of $D \setminus \eta_0$ equals the area of $D_1$ 
and $D_2$ respectively, and there exists $g \in \homeo_{\lambda}(S^2)$ with support on $D$, such that $g(\eta_1) = \eta_0$. The composition $gh$ will now 
transport area in a definite way from $D_1$ to $D_2$, as required. To conclude the proof, note that this construction produces a definite amount of area transport on the one hand, 
but also, by choosing the perturbations smaller than the size of $D$ if necessary, the construction produces sufficiently small area transport, to produce any image arc, as required.
\end{proof}

\begin{proof}[Proof of Proposition~\ref{thm_area_max}]
Since every orientation preserving homeomorphism of the sphere is isotopic to the identity, first using a finite number of locally supported 
homeomorphisms from $\homeo_{\lambda}(S^2)$ moving the vertices of the warped grid $h(\Gamma(\epsilon))$, with $h$ the homeomorphism that needs to be approximated,
back into their original position, then use finitely many elementary moves constructed in Lemma~\ref{lem_elem_move} to move back all the arcs of the still warped grid 
back to the original grid $\Gamma(\epsilon)$. 
\end{proof}

Next, we prove the following.

\begin{prop}\label{thm_ant_homeo}
The arrow $\Ant_{\lambda}(S^2) \lra \homeo(S^2)$ is complete, that is, a closed group $G$ extending $\Ant_{\lambda}(S^2)$ containing homeomorphisms not preserving the 
antipodal action and not preserving area, equals $\homeo(S^2)$.
\end{prop}

The difference with the case $\homeo_{\lambda}(S^2) \lra \homeo(S^2)$ is that now the small perturbation domains have to be constructed with the more restricted 
group $\Ant_{\lambda}(S^2)$, rather than $\homeo_{\lambda}(S^2)$. Let $\Ant_{\lambda}(S^2) \subset G \subseteq \homeo(S^2)$, where $G$ is not contained in 
either $\homeo_{\lambda}(S^2)$ or $\Ant(S^2)$. Take a homeomorphism $h \in \homeo(S^2)$, which we may again assume preserves Lebesgue null sets, 
and choose $\epsilon>0$. By a finite composition of elementary moves in $G$, we move the warped grid $h(\Gamma(\epsilon))$ back to the original grid $\Gamma(\epsilon)$, 
for each $\epsilon$, as before. The idea of the proof is that, we can construct perturbations supported on two small disks, both contained in one fundamental domain of $\R\PPP^2$ 
of the sphere. One disk acts as a good perturbation disk and the other disk acts as a trashcan disk supporting the bad perturbation.

\begin{lem}[Elementary move]\label{lem_elem_move_ii}
For each sufficiently small $\delta>0$, and for each pair of disjoint good disks $D_1, D_2 \subset S^2$ in the same fundamental domain of $\R\PPP^2$ and of area at most $\delta$,
and for each pair of good arcs $\eta_1, \eta_2 \subset D_1$ having the same endpoints in $\partial D_1$, there exists $g \in G$ with the property that $g$ is supported on 
$D_1 \cup D_2$, and $g(\eta_1) = \eta_2$.
\end{lem}

\begin{figure}[h]
\begin{center}
\psfrag{p_2}{$q$}
\psfrag{p_1}{$\infty$}
\psfrag{p_0}{$0$}
\psfrag{g_1}{$\eta_1$}
\psfrag{g_2}{$\eta_2$}
\psfrag{h_0}{$g$}
\psfrag{D}{$D_1$}
\psfrag{D_1}{$D_2$}
\includegraphics[scale=0.30]{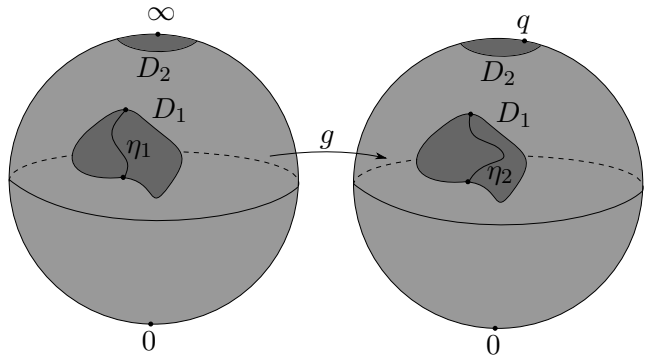}
\caption{Elementary move of Lemma~\ref{lem_elem_move_ii}.}
\label{fig_lemma_elem_move}
\end{center}
\end{figure}

\begin{lem}\label{lem_local_variation}
There exists $h \in G$ with $h(\infty) = \infty, h(0) \neq 0$ and such that for all sufficiently small neighborhoods $U_0, U_{\infty} \subset S^2$ of $0, \infty$ respectively, 
we have that $V_0 := h(U_0)$ and $V_{\infty} := h(U_{\infty})$ do not contain antipodal points and the area gradient on at least one of $U_0$ or $U_{\infty}$ is non-trivial.
\end{lem}

\begin{proof}
Suppose not, then for each $g \in G$ and for each pair $\{p,-p\} \in S^2$ mapped to $q_1$ and $q_2 \neq -q_1$, there exist neighborhoods $U_{p}$ and $U_{-p}$ 
such that the Lebesgue measure of each measurable set in each of these disks is multiplied by a constant depending only on the basepoint $p$ and $-p$. Passing 
to smaller $U_p$ and $U_{-p}$ if necessary, we may assume that $V_p \cap V_{-p} = \emptyset$, with $V_p := g(U_p)$ and $V_{-p} := g(U_{-p})$. 
Then the same conclusion must hold if $q_1$ and $q_2 = -q_1$ are antipodal. Indeed, if there would exist $h \in G$ and points $p,-p$ sent to $q,-q$ for which for 
arbitrarily small neighborhoods $U_{p}$ and $U_{-p}$ of $p$ and $-p$, small disks of equal area are sent to disks of differing area in the image $V_{p}$ and $V_{-p}$, 
then the composition $h g$ has the property of having non-trivial area gradient on a small neighborhood nested around the antipodal points whose images 
$W_p = hg(U_p)$ and $W_{-p} = hg(U_{-p})$ under $hg$ are by construction not antipodal in the sense that $-W_p \cap W_{-p} = \emptyset$. 
Now, take any $g \in G$ and consider the function that to $p \in S^2$ assigns the multiplier of the areas, which by the argument above exists for each $p \in S^2$. 
Since this function, by its construction, is constant on each small enough open neighborhood that exists for each point, and since finitely many of these neighborhoods cover 
the sphere by compactness, observing that the function has to be constant on overlaps, it has to be constant on the entire sphere, and thus 
equal to $1$. Thus $g \in \homeo_{\lambda}(S^2)$ and thus $G \subseteq \homeo_{\lambda}(S^2)$, a contradiction. 
\end{proof}

\begin{proof}[Proof of Lemma~\ref{lem_elem_move_ii}]
By Lemma~\ref{lem_local_variation}, there exists $h \in G$ with the properties that $h(\infty)=\infty$ and $h(0) = q \neq 0$ and arbitrarily small neighborhoods $U_0$ 
and $U_{\infty}$ containing $0$ and $\infty$ respectively, so that the images $V_0$ and $V_{\infty}$ under $h$ are disjoint, contained in one fundamental domain 
and $h$ distorts the area of $V_0$ non-trivially. Using elements of $\Ant_{\lambda}(S^2)$ that have support on $U_0$ and $U_{\infty}$, 
using that the area-distortion on $V_0$ is non-trivial and using the argument in the proof of Lemma~\ref{lem_elem_move}, then conjugating by elements of $\Ant_{\lambda}(S^2)$ 
to send the neighborhoods $V_0$ and $V_{\infty}$ to arbitrary domains in the same fundamental domain of the sphere $S^2$, one constructs homeomorphisms 
in $G$ that have support contained in domains with small area, and that preserve a good disk $D_1$ in one of these domains that can send one arc $\eta_1 \subset D_1$ to another 
arc $\eta_2 \subset D_2$ with the same endpoints. 
\end{proof}

\begin{proof}[Proof of Proposition~\ref{thm_ant_homeo}]
Using the elementary move of Lemma~\ref{lem_elem_move_ii}, we can make perturbations supported on two disks, one as small as desired compared to the other.
Namely, by decomposing a homeomorphism on a larger disk into finitely many homeomorphisms supported on smaller disks, fixing the trash can disk throughout, 
this can be constructed. Using the elementary moves, and applying a diagonal argument due to the arbitrarily small trashcan error, one moves the edges of the 
distorted grid $h(\Gamma(\epsilon))$ back into their original position, as in the previous case, to obtain $\Gamma(\epsilon)$, for each $\epsilon>0$, as required.
\end{proof}

Finally, we discuss several arrows to do with antipodal action groups.

\begin{prop}\label{thm_ant_ant}
The arrow $\Ant_{\lambda}(S^2) \lra \Ant(S^2)$ is complete.
\end{prop}

\begin{proof}
This follows since these groups pass to the quotient to closed groups $\homeo_{\lambda}(\R \PPP^2)$ and $\homeo(\R\PPP^2)$ of the closed surface $\R \PPP^2$,
where the same argument as the proof of Proposition~\ref{thm_area_max}.
\end{proof}

Furthermore, the proof of Proposition~\ref{thm_ant_homeo} yields a proof of the following.

\begin{prop}\label{thm_ant_max}
The arrow $\Ant(S^2) \lra \Homeo(S^2)$ is complete, that is, the group $\Ant(S^2)$ is maximal in $\homeo(S^2)$.
\end{prop}

\begin{proof}
Since $\Ant_{\lambda}(S^2)$ is contained in $\Ant(S^2)$, the result follows from Proposition~\ref{thm_ant_homeo}.
\end{proof}

\subsection{Proof of Theorem A}

The proof consists in summarizing the above results. To wit, completeness of the arrows
\[ \Rot(S^2) \lra \Mobius(S^2) ~\tu{and}~ \Rot(S^2) \lra \Lin(S^2) \] 
is proved in Proposition~\ref{prop_rot_lin} and Proposition~\ref{prop_rot_mob_max}. where maximality of $\Ant(S^2)$ in $\homeo(S^2)$ 
follows from Proposition~\ref{thm_ant_max}. Completeness of the arrows 
\[ \homeo_{\lambda}(S^2) \lra \homeo(S^2), ~\Ant_{\lambda}(S^2) \lra \Ant(S^2)~\tu{and}~\Ant_{\lambda}(S^2) \lra \homeo(S^2) \] 
is proved in Proposition~\ref{thm_area_max}, Proposition~\ref{thm_ant_homeo}, and Proposition~\ref{thm_ant_ant}, respectively. 
The following set of intersections are all trivial, in the sense that their intersection equals the rotation group $\Rot(S^2)$. These are
\[ \Mobius(S^2) \cap \Lin(S^2), ~\Mobius(S^2) \cap \homeo_{\lambda}(S^2), ~\Mobius(S^2) \cap \Ant_{\lambda}(S^2) \]
and
\[ \Mobius(S^2) \cap \Ant(S^2), ~ \Lin(S^2) \cap \Ant_{\lambda}(S^2), ~ \Lin(S^2) \cap \Homeo_{\lambda}(S^2) \]
Indeed, these claims follow from combining Lemma~\ref{lem_conf_lin}, stating that $\Ant(S^2) \cap \Mobius(S^2) = \Rot(S^2)$, with the inclusions $\Lin(S^2) \subset \Ant(S^2)$ 
and $\Ant_{\lambda}(S^2) \subset \Ant(S^2)$, with Lemma~\ref{lem_area_linear} stating that $\Lin(S^2) \cap \Homeo_{\lambda}(S^2) = \Rot(S^2)$, with the 
inclusion $\Ant_{\lambda}(S^2) \subset \Homeo_{\lambda}(S^2)$, and that $\Mobius(S^2) \cap \homeo_{\lambda}(S^2) = \Rot(S^2)$ was proved 
in Lemma~\ref{lem_mob_area_intersect}. This completes the proof of Theorem A.

\section{The case of higher genus surfaces}

In this final section we list transformation groups on higher genus surfaces where the topology of the surface obstructs the finite-dimensional
actions of the sphere, but allows for mapping class groups to act non-trivially. 

\subsection{Homeomorphism groups of the torus}

Denoting $\homeo(\T^2)$ the full group of torus homeomorphisms, with $\T^2$ the torus equipped with the Euclidean metric, in a way similar to the case of the sphere $S^2$, 
a group $G \subset \homeo(\T^2)$ is said to be homogeneous if $G$ contains the translation group $\Trans(\T^2)$. In terms of the above described Lie groups, let us now 
consider the subgroups of the homeomorphism group of the torus $\T^2$. Denote $\homeo_{\Id}(\T^2) \subset \homeo(\T^2)$ the subgroup of homeomorphisms homotopic 
to the identity on the torus.

\subsubsection{Holomorphic mappings}

The group of holomorphic mappings of $\T^2$, or alternatively the group of isometries of $\T^2$, equals the translation group $\Trans(\T^2)$, 
analogous to the rotation group $\Rot(S^2)$ of the sphere $S^2$. This group is minimal in $\homeo(\T^2)$ as a closed and transitive group.

\subsubsection{Invariant foliations}

Consider a group $G \subset \homeo(\T^2)$ that preserves a foliation of essential simple closed curves in the torus $\T^2$, which after conjugation
we may suppose to be the foliation homeomorphic to the horizontal foliation. In this case, $G$ is conjugate to the group $\Skew(\T^2)$ of skew-product 
homeomorphisms of the form $g(x,y) = (\varphi(x), \psi(x,y))$, with $\varphi \colon \mathbb{S}^1 \ra \mathbb{S}^1$ and $\psi \colon \T^2 \ra \mathbb{S}^1$,
where the genuinely horizontal foliation is preserved. In case $G$ leaves invariant two transverse foliations by essential simple closed curves, then $G$ is 
conjugate to the group $\Prod(\T^2)$ of product homeomorphisms of the form $g(x,y) = (\varphi(x), \psi(y))$. In case, the group $G$ leaves invariant three 
pairwise transversal foliations, each of which is a foliation by essential simple closed curves, is conjugate to the translation group $\Trans(\T^2)$.

\subsubsection{Area-preserving homeomorphisms} 

The group $\homeo_{\lambda, \Id}(\T^2) \subset \homeo_{\Id}(\T^2)$ of area-preserving homeomorphisms homotopic to the identity, is a 
closed and transitive subgroup of $\homeo_{\Id}(\T^2)$ and is a maximal subgroup of $\homeo_{\Id}(\T^2)$, as it is maximal in the case of the sphere
by Theorem A.

\subsubsection{Mapping class group} 

Denote $H \subset \MCG(\T^2)$ a proper subgroup of the mapping class group. Define $\homeo_H(\T^2) \subset \homeo(\T^2)$ the subgroup of homeomorphisms 
whose homotopy classes lie in the subgroup $H \subset \MCG(\T^2) \cong \SL(2, \Z)$ that represents $\MCG(\T^2)$ by area-preserving homeomorphisms. Each 
such group $\homeo_{H}(\T^2)$ is a proper closed and transitive subgroup of $\homeo(\T^2)$ due to global topological obstructions. 

\subsubsection{Finite degree covers} 

Define $\homeo_{n,m}(\T^2) \subset \homeo(\T^2)$ the subgroup of homeomorphisms commuting with a finite degree $(n,m)$ regular covering action of the torus, 
defined as follows. Given a homeomorphism $h \in \homeo(\T^2)$, with $\T^2 = \R^2 \slash \Z^2$, it naturally induces a homeomorphism $h_{n,m}$ on the torus 
$\T^2_{n,m} := \R^2 \slash (n\Z \times m\Z)$, which itself is homeomorphic to $\T^2$. Each closed and transitive subgroup of $\homeo(\T^2)$ of homeomorphisms 
that commute with respect to each such non-trivial pair $(n,m) \neq (1,1)$ defines a proper closed and transitive subgroup of $\homeo(\T^2)$. Concrete examples 
are the product groups 
\begin{equation}
G_{n,m} := \homeo_{n}(\mathbb{S}^1) \times \homeo_m(\mathbb{S}^1),
\end{equation}
with $\homeo_k(\mathbb{S}^1)$ the closed and transitive group of circle homeomorphisms commuting with rotations of finite order $k \in \N$. 

\subsection{Homeomorphism groups of higher genus surfaces}

In the case of the homeomorphism group $\homeo(S)$ of a genus at least two, several interesting closed transitive subgroups arise, in particular those 
arising as subgroups of the mapping class group and regular finite degree covering actions. There is a moduli space of hyperbolic metrics on a given topological 
surface of genus at least two. A hyperbolic metric can be chosen on a surface of genus at least two, a different choice of hyperbolic metric gives rise to conjugate 
homeomorphism groups. Furthermore, unlike the sphere and torus, a closed surface $S$ of genus at least two admits only a finite group of isometries, and consequently 
this group is not transitive.

\subsubsection{Area-preserving mappings} 

The group $\homeo_{\lambda}(S)$ of area-preserving homeomorphisms, with the subgroup of all area-preserving homeomorphisms isotopic to the identity being a 
maximal subgroup of the group $\homeo_{\Id}(S)$ of homeomorphisms of $S$ isotopic to the identity, and where every homotopy class of a homeomorphism 
admits an area-preserving representative creating the full group $\homeo_{\lambda}(S)$ of area-preserving homeomorphisms.

\subsubsection{Mapping class groups} 

Similar to the torus case, denote $H \subset \MCG(S)$ a subgroup of the full mapping class group $\MCG(S)$ of a closed surface, see~\cite{FM} for a good exposition. 
Write $\homeo_H(S) \subset \homeo(S)$ the subgroup of homeomorphisms whose homotopy classes lie in the subgroup $H \subset \MCG(S)$. 
The group $\homeo_{H}(S) \subset \homeo(S)$ is a closed and transitive group~\cite{fisher} that is a proper subgroup of $\homeo(S)$ due to global topological obstructions. 
First, by a famous result of Kerckhoff~\cite{kerck} (the Nielsen realization problem), every finite subgroup of $\MCG(S)$ is realized by a finite group of conformal homeomorphisms 
of the surface relative to some hyperbolic metric on $S$, which by Teichm\"uller theory is not unique, and hence the group $\Fin(S) \subset \MCG(S)$ of all finite order isotopy classes
is induced by conformal homeomorphisms. For example one can take $\Fin(S)$ and add a single Dehn-twist to generate a larger group. Or one can take the Torelli subgroup 
$\Torelli(S) \subset \MCG(S)$ of isotopy classes acting trivially on homology and a single non-Torelli element, to produce larger subgroups. The question is to map the internal 
structure of the mapping class group by finding all possible relations between subgroups and showing completeness of these arrows.

\subsubsection{Finite degree covers} 

Given closed surfaces $S_1$ and $S_2$ of higher genus, consider a regular finite degree covering mapping $\pi \colon S_1 \ra S_2$ projecting the surface $S_2$ of genus $g_2$ 
onto another closed surface $S_1$ of genus $2 \leq g_1 \leq g_2$. Given a closed and transitive group $G_1 \subseteq \homeo(S_1)$, it lifts to its regular covering space to a 
group of homeomorphisms $G_2 \subset \homeo(S_2)$ that is also closed and transitive, and furthermore, it is a proper subgroup of $\homeo(S_2)$. 
Therefore, given the collection of all such regular covering mappings from a given surface $S$ to its quotient surfaces, combined with the closed and transitive subgroups of the 
homeomorphism group defined on these quotient surfaces, gives rise in a natural way to a list of closed and transitive proper subgroups of the homeomorphism group of the 
given surface $S$.

\end{document}